\documentclass{article}
\usepackage{times}
\usepackage[hyperindex=true,pageanchor=true,hyperfigures=true,backref=false]{hyperref} 
\usepackage{amsmath}
\usepackage{amssymb}
\usepackage{amsthm}
\usepackage[capitalise,nameinlink]{cleveref}
\usepackage{mathtools}
\usepackage{tikz-cd}
\usepackage{comment}

\usepackage{enumitem}

\usepackage{url}
\usepackage{dsfont} 
\usepackage[toc]{appendix}
\DeclareSymbolFontAlphabet{\Bbb}{AMSb}
\usepackage[T1]{fontenc}

\usepackage{mleftright,xparse}

\usepackage{thm-restate,amsthm}

\renewcommand{\i}{\mathrm{i}}
\renewcommand{\d}{\, \textup{d}}
\newcommand{\T}{\mathbb{T}}

\newlength{\myleftmargin}
\usepackage{amsmath}
\usepackage{amssymb}
\usepackage{amsthm}
\usepackage{color}
\DeclareSymbolFontAlphabet{\Bbb}{AMSb}

\newtheorem{theorem}{Theorem}[section]
\newtheorem{lemma}[theorem]{Lemma}

\newtheorem{corollary}[theorem]{Corollary}
\newtheorem{definition}[theorem]{Definition}

\newlength{\fixboxwidth}
\newcommand{\mycdot}{\,\cdot\,}

\newcommand{\N}{\mathbb{N}}
\newcommand{\C}{\mathbb{C}}

\newcommand{\R}{\mathbb{R}}
\newcommand{\Z}{\mathbb{Z}}

\newcommand{\E}{\mathbb{E}}

\newcommand{\e}{\mathrm{e}}

\DeclareMathOperator*{\esssup}{ess\,sup}
\DeclareMathOperator*{\essinf}{ess\,inf}

\title{Minimal-Norm Extensions of Stationary Kernels on Subgroups of Locally Compact Abelian Groups and Gaussian Conditioning}

\author{Daniel Winkle\\
School of Engineering Sciences, LUT University (Lappeenranta, Finland)\\
\texttt{\small Daniel.Winkle@lut.fi}\\
}

\begin{document}

\maketitle

\begin{abstract}
We study restrictions of stationary kernels on locally compact abelian groups $G$ to closed subgroups $H$. For a nonnegative spectral density $\hat k$, we derive an explicit fibrewise Fourier representation of the minimal-norm extension operator from the reproducing kernel Hilbert space of the restricted kernel on $H$ to the original space on $G$. We characterize when the canonical Fourier formula extends boundedly from $L^2(H)$ to $L^2(G)$, identify its exact operator norm and lower norm, and obtain bounds on the associated interpolation spaces. When $G$ is compact, the extension is a contraction and, for stationary Gaussian random variables admitting a measurable continuous version, maps the observed restriction to the conditional expectation. We also give a counterexample to a previously asserted supremum-norm contraction and illustrate the theory through cardinal interpolation and conditioning on one-dimensional subgroups of the torus.
\end{abstract}

\textbf{Mathematical Subject Classification (2020).} Primary 60G15; Secondary 46E22

\textbf{Key Words.} Gaussian random variables, Kernel methods, Conditioning

\section{Introduction}\label{sec:intro} 

Conditioning Gaussian random variables on observations is a fundamental operation in probability theory, statistics, and related approximation problems \cite{bogachev1998gaussian,steinwart2024conditioning}. In the setting of Gaussian processes, this operation appears, for instance, in kriging and Gaussian process regression, and it is closely related to kernel interpolation \cite{wenzel2026sharp,rasmussen2006gaussian,wendland2004scattered}. If only finitely many observations are used, the conditional mean and covariance can be described by finite-dimensional linear algebra involving covariance matrices \cite{rasmussen2006gaussian}. However, in many situations one is naturally led to observations on infinitely many points, for example when a process is observed on a subdomain, a submanifold, or a subgroup \cite{lagatta2013continuous,steinwart2024conditioning,fuselier2012scattered}. In this case, the conditional expectation is naturally studied as an operator between function spaces rather than as a finite-dimensional matrix \cite{lagatta2013continuous,owhadi2018conditioning,steinwart2024conditioning,travelletti2024disintegration}.

The present paper studies this operator on the level of reproducing kernel Hilbert spaces (RKHS) associated with stationary kernels on locally compact abelian groups. Let $G$ be a locally compact abelian group and let $H\subseteq G$ be a closed subgroup \cite{Folland}. Starting from a stationary kernel $k$ on $G$, or equivalently from an RKHS described by a spectral density, we consider its restriction to $H$ and the corresponding RKHSs $\mathcal N(G)$ and $\mathcal N(H)$ \cite{feichtinger2007minimal,wendland2004scattered}. The natural object in this setting is the minimal-norm extension operator
\begin{align*}
    M_W:\mathcal N(H)\to \mathcal N(G),
\end{align*}
which maps a function on $H$ to its minimal-norm extension in $\mathcal N(G)$ \cite[Chapter 10.7]{wendland2004scattered}. Using the Fourier representation of stationary kernels and the structure of closed subgroups of locally compact abelian groups, we derive an explicit formula for this operator in terms of the spectral density of $k$ \cite{feichtinger2007minimal}. In the compact case, this RKHS construction has a direct probabilistic interpretation: if $X$ is a stationary Gaussian random variable on $G$ and $Y=X|_H$, then the corresponding extension operator maps the observation $Y$ to the conditional expectation $\mathbb E(X\mid Y)$ \cite{WinkleAnalysis,steinwart2024conditioning}.

The present work is also motivated by a boundedness issue in infinite-dimensional Gaussian conditioning. In \cite{lagatta2013continuous}, continuous disintegrations of Gaussian processes are studied, and a contraction property is asserted for the corresponding conditional expectation operator between spaces of continuous functions. We show that this contraction statement does not hold in the stated generality. More precisely, using the Gaussian kernel and the fact that the corresponding sample paths are analytic, we construct a counterexample for which the conditional expectation operator is given by analytic continuation and cannot satisfy the claimed norm bound. This illustrates that, for infinite-dimensional observations, boundedness of the conditioning operator is a delicate question and depends strongly on the chosen function spaces. 

The other main contributions of this paper are as follows. First, we derive an explicit Fourier representation of the minimal-norm extension operator for restrictions of stationary kernels to closed subgroups of locally compact abelian groups. Second, we give conditions on the spectral density under which this operator extends boundedly from $L^2(H)$ to $L^2(G)$, and we identify the corresponding operator norm. Third, in the compact case, we interpret this extension operator as the conditional expectation operator for stationary Gaussian random variables observed on H. Finally, we illustrate the theory in examples including cardinal interpolation and conditioning on one-dimensional subgroups of the torus. To the best of our knowledge, the $L^2$-boundedness of the subgroup extension operator in terms of the spectral density has not been treated in this generality.

A related but different line of work concerns the effect of group symmetries and invariances in kernel methods. For instance, \cite{Behrooz} studies kernel ridge regression on compact manifolds under invariance with respect to compact Lie group actions and quantifies the corresponding gain in sample complexity. While their focus is statistical learning with invariant target functions, the present paper studies conditioning and minimal-norm extension operators arising from restrictions to closed subgroups. Both viewpoints illustrate how group structure can be used to reduce or reorganize kernel-based problems.

The paper is organized as follows. Section \ref{sec:2} recalls the necessary background on locally compact abelian groups, stationary kernels, RKHSs, and Gaussian random variables. Section \ref{sec:results} contains the main results. We first state the Fourier representation of the minimal-norm extension operator for stationary kernels on locally compact abelian groups and then specialize the result to compact groups, where it admits a direct interpretation in terms of conditional expectations of stationary Gaussian random variables. Section \ref{sec:4} discusses examples. In particular, we provide a counterexample to a contraction statement from \cite{lagatta2013continuous}, and we treat cardinal interpolation on $h\mathbb Z\subset\mathbb R$ as well as conditioning on one-dimensional subgroups of the torus. The proofs of the main results are given in Appendix \ref{sec:proofs}, while auxiliary statements are collected in Appendix \ref{sec:appendix}.

\color{black}

\section{Preliminaries}\label{sec:prelim}\label{sec:2} 
This section collects the background material needed for the main results. We begin with locally compact abelian groups, then discuss stationary kernels and their associated reproducing kernel Hilbert spaces, and finally recall the relevant notions concerning Gaussian random variables.

\subsection{Locally Compact Abelian Groups}
We briefly recall the notions from harmonic analysis that are needed in the sequel. For a more detailed account, we refer to \cite{Folland}.

Let $G$ be a locally compact abelian group with neutral element denoted by $0$. We write $\hat{G}$ for the dual of $G$, that is, the group of continuous characters $\xi: G\to \mathbb T$ with $\T:=\{ z \in \C \, \vert \, \vert z \vert=1\}$. Moreover, $C_c(G)$ denotes the space of continuous functions on $G$ with compact support.

Given a closed subgroup $H \subseteq G$, we denote by $H^\perp$ its annihilator, defined by 
\begin{align*}
    H^\perp:=\{ \xi \in \hat{G} \, \vert \, \xi(h)=1, \, \, \forall h \in H\}.
\end{align*}
We fix Haar measures on $G$ and $\hat{G}$ such that the Fourier inversion formula holds, see Theorem \ref{theorem:Fourierinversion}. To keep the notation light, we do not introduce separate symbols for these measures. The measure used in an integral will always be clear from the domain of integration.

For $f \in L^1(G)$, we define its Fourier transform $\hat{f}:\hat{G} \to \C$ by
\begin{align*}
    \hat{f}(\xi)= \int_G \overline{\xi(x)} f(x) \d x,
\end{align*}
see \cite[Chapter 4.2]{Folland}. In the case of $G=\mathbb{R}$, we identify $\hat{G}=\mathbb{R}$, and the characters $\xi: \R \to \T$ are given by $\xi(x)=\mathrm{e}^{2\pi \i \xi x}$. Thus, we recover the usual Fourier transform.
Moreover, we have the following inversion formula.
\begin{theorem}\label{theorem:Fourierinversion}
    Let $f \in L^1(G)$ and assume that $\hat{f}\in L^1(\hat{G})$. Then 
    \begin{align*}
        f(x)= \int_{\hat{G}} \xi(x) \hat{f}(\xi) \d \xi
    \end{align*}
for almost all $x \in G$. Moreover, if $f$ is continuous, then the identity holds for every $x \in G$.
\end{theorem}
For a proof we refer to \cite[Theorem 4.32]{Folland}.  
Recall that the Fourier transform can also be extended onto $L^2(G)$.
\begin{theorem}[Plancherel]\label{theorem:Plancherel}
    The Fourier transform extends uniquely to an isometric isomorphism from $L^2(G)$ onto $L^2(\hat{G})$.
\end{theorem}
For a proof, see \cite[Theorem 4.25]{Folland}.

\begin{lemma}\label{lemma:compactdiscrete}
If $G$ is compact, then $\hat{G}$ is discrete.
\end{lemma}
For a proof, see \cite[Proposition 4.4]{Folland}.

We shall also use the standard duality relations for closed subgroups of locally compact abelian groups. If $H\subseteq G$ is closed, then Pontryagin duality identifies
\begin{align*}
    \hat H \cong \hat G/H^\perp,
\qquad
\widehat{G/H} \cong H^\perp,
\end{align*}
see \cite[Theorem 4.39]{Folland}. Thus elements of $\hat H$ may be regarded as cosets in $\hat G/H^\perp$. For notational simplicity, we write such a coset as $\xi_H$. If $\xi\in\hat G$ is a representative of the coset $\xi_H=\xi+H^\perp$, then expressions of the form
\begin{align*}
    \xi_H+\xi^\perp,
\qquad \xi^\perp\in H^\perp,
\end{align*}
are understood as shorthand for $\xi+\xi^\perp$. Whenever such notation appears inside an integral over $H^\perp$, the value of the integral is independent of the chosen representative $\xi$, by translation-invariance of Haar measure on $H^\perp$.

The following lemma fixes the Haar-measure normalization that will be used throughout the paper.
\begin{lemma}\label{lem:Fubini}
Given a locally compact abelian group $G$ and a closed subgroup $H\subseteq G$, there exist measures on $\hat{G}/H^\perp$ and $H^\perp$ such that for all $f \in L^1(G)$ with $\hat{f} \in C_c(\hat{G})$ we have for almost all $x \in G$
\begin{align*}
f(x)
    =\int_{\hat{G}} \hat{f}(\xi) \cdot \xi(x) \d \xi 
    &= \int_{\hat{G}/H^\perp} \int_{H^\perp} \hat{f}(\xi_H+\xi^\perp) \cdot (\xi_H+\xi^\perp)(x) \d \xi^\perp \d \xi_H \\
    &= \int_{H^\perp} \int_{\hat{G}/H^\perp} \hat{f}(\xi_H+\xi^\perp) \cdot (\xi_H+\xi^\perp)(x) \d \xi_H \d \xi^\perp.
\end{align*}
Moreover, if $f \in L^2(G)$ we have 
\begin{align*}
    f
    =\int_{\hat{G}} \hat{f}(\xi) \cdot \xi \d \xi 
    &= \int_{\hat{G}/H^\perp} \int_{H^\perp} \hat{f}(\xi_H+\xi^\perp) \cdot (\xi_H+\xi^\perp) \d \xi^\perp \d \xi_H \\
    &= \int_{H^\perp} \int_{\hat{G}/H^\perp} \hat{f}(\xi_H+\xi^\perp) \cdot (\xi_H+\xi^\perp) \d \xi_H \d \xi^\perp.
\end{align*}
\end{lemma}
For $\hat{f} \in C_c(\hat{G})$ a proof can be found in \cite[Theorem 2.49]{Folland} for the first line of equalities; the last equality follows by Fubini. The statement for $f \in L^2(G)$ is a consequence of the Plancherel theorem, see \cite[Theorem 4.25]{Folland}. 

Throughout the paper, we fix Haar measures on $H^\perp$ and $\hat G/H^\perp$ satisfying Lemma \ref{lem:Fubini}. With this normalization, the Fourier inversion formula can be decomposed along the cosets of $H^\perp$.

\subsection{Stationary Kernels}

We recall the definition of stationary kernels used in this paper. The assumptions imposed below are slightly stronger than in the most general formulation of Bochner's theorem, see \cite[Theorem 4.18]{Folland}. They ensure that the covariance function is continuous and that the associated reproducing kernel Hilbert space can be described by Fourier methods.

\begin{definition}
We call a function $k:G \times G \to \R$ a kernel if there exists a Hilbert space $H_0$ and a feature map $\Phi_0:G \to H_0$ such that 
\begin{align*}
    k(x,y)=\langle \Phi_0(x),\Phi_0(y)\rangle_{H_0}, \qquad  \forall x,y \in G.
\end{align*}
We call a kernel $k$ stationary if there exists a function $\hat{k} \in L^1(\hat{G}) \cap L^\infty(\hat{G})$ with $\hat{k}: \hat{G} \to \R $ such that $\hat{k} \geq 0$ almost everywhere and 
\begin{align*}
    k(x,y)= \int_{\hat{G}} \hat{k}(\xi) \xi(x  -y) \d \xi, \qquad \forall x,y \in G.
\end{align*}
\end{definition} 
By the assumption that $\hat{k} \in L^1(\hat{G})$ we have that the kernel $k$ is continuous, see \cite[Lemma 4.28]{SteinwartSVM}.  
We use the notation $\hat{k}$ for the spectral density, although it should not be confused with the Fourier transform of $k$ as a function on $G \times G$.

Associated with $k$ is an RKHS, which we denote by $\mathcal{N}(G)$. If $k$ is a kernel, the RKHS is given by 
\begin{align*}
    \mathcal{N}(G):= \{ f : G \to \R \, \vert \, \exists w \in H_0 \, \, \textup{with} \, \, f(x)=\langle w, \Phi_0(x) \rangle_{H_0} \, \, \textup{for all} \, \, x \in G \}.
\end{align*}
It is equipped with the norm 
\begin{align*}
    \Vert f \Vert_{\mathcal{N}(G)}:= \inf \{ \Vert w \Vert_{H_0} \, \vert \, w \in H_0 \, \, \textup{with} \, \, f(x) = \langle w , \Phi_0(x) \rangle_{H_0} \, \, \textup{for all} \, \, x \in G  \},
\end{align*}
see \cite[Theorem 4.21]{SteinwartSVM}. 

We now focus on stationary kernels $k$ on $G$. Under the assumptions imposed above, the associated RKHS satisfies $\mathcal{N}(G) \subseteq C_b(G) \cap L^2(G)$, see Lemma \ref{lem:RKHSProperties}. Here $C_b(G)$ denotes the space of bounded and continuous functions on $G$. In particular, elements of $\mathcal{N}(G)$ may be Fourier transformed in the $L^2(G)$-sense by the Plancherel theorem \ref{theorem:Plancherel}. Moreover, for $ f\in \mathcal{N}(G)$, the Fourier transform satisfies $\hat{f} \in L^1(\hat{G}) \cap L^2(\hat{G})$ again by Lemma \ref{lem:RKHSProperties}. 

Restricting $k$ to the closed subgroup $H$ yields the kernel $k_{H}:H \times H \to \R$. We denote the associated RKHS by $\mathcal{N}(H)$, see \cite[Chapter 10.7]{wendland2004scattered}. By the restriction and extension theory of RKHSs, every $f \in \mathcal{N}(H)$ admits a unique minimal-norm extension to $\mathcal{N}(G)$. We denote the corresponding extension operator by $M_W:\mathcal{N}(H) \to \mathcal{N}(G)$, where 
\begin{align}\label{eq:condition}
    M_W f := \textup{argmin} \{ \Vert f_G\Vert_{\mathcal{N}(G)} \, \vert \, f_G \in \mathcal{N}(G) \, \, \textup{such that} \, \, f_G|_H=f \}.
\end{align}
For this construction, see \cite[Theorem 10.46]{wendland2004scattered}. 

For the restricted kernel on a closed subgroup $H\subseteq G$, we also need to ensure that the associated RKHS is contained in $L^2(H)$, so that the Fourier transform on $H$ is well-defined. 
Note that the spectral density $\hat{k}_H:\hat{H} \to \R$ is given by 
\begin{align*}
    \hat{k}_H(\xi_H) = \int_{H^\perp} \hat{k}(\xi_H+\xi^\perp) \d \xi^\perp,
\end{align*}
see Lemma \ref{lem:INL2}.
\color{black}
If $G$ is not compact, we therefore impose the additional standing assumption 
\begin{align}\label{eq:ConditionPrelim}
    \sup_{\xi_H \in \hat{G}/H^\perp} \int_{H^\perp} \hat{k}(\xi_H + \xi^\perp) \d \xi^\perp < \infty.
\end{align}
By Lemma \ref{lem:INL2}, this condition implies $\mathcal{N}(H) \subseteq L^2(H)$. Consequently, elements of $\mathcal{N}(H)$ may also be Fourier transformed in the $L^2(H)$-sense.

The main result, Theorem \ref{Theorem:MainResult}, gives an explicit formula for the operator $M_W$ and provides conditions under which it extends to a bounded operator $M:L^2(H) \to L^2(G)$. Since $M_W:\mathcal{N}(H) \to \mathcal{N}(G)$ is bounded by construction, see  \cite[Theorem 10.46]{wendland2004scattered}, the $L^2$-boundedness of its extension also yields, by interpolation space theory, boundedness between the corresponding interpolation spaces associated with $\mathcal{N}$ and $L^2$, see \cite{bergh1976interpolation}.

We will only use interpolation spaces in a simple Hilbert-space setting. Namely, under the Fourier-side identification of $\mathcal{N}(G)$, the spaces obtained by real interpolation between $\mathcal{N}(G)$ and $L^2(G)$, with interpolation exponent $q=2$, can be represented as weighted $L^2$-spaces, see \cite[Section 5]{bergh1976interpolation}. More precisely, for  $0 \leq \theta \leq 1$, we define
\begin{align*}
    \mathcal{N}(G)_{\theta}:= \left\{ f  \in L^2(G) \, \middle\vert \, \int_{\hat{G}} \frac{\vert \hat{f} \vert^2}{\vert\hat{k}\vert^{1-\theta}} (\xi) \, \d \xi < \infty \right\},
\end{align*} 
equipped with the corresponding norm. 
If $\hat{k}$ vanishes on a set of positive Haar measure, we use the convention that the quotient $\hat{f}/\hat{k}$ is infinite unless $\hat{f}=0$ almost everywhere on $\hat{k}=0$.
Analogously, $\mathcal{N}(H)_\theta$ is defined using the spectral density of the restricted kernel on $H$.

\begin{corollary}\label{cor:Interpolation}
Given two bounded operators $T_0:\mathcal{N}(H) \to \mathcal{N}(G)$ and $T_1:L^2(H) \to L^2(G)$ such that $T_0 f =T_1f$ for all $f \in \mathcal{N}(H)$, where the equality is understood in $L^2(G)$, there exists a bounded operator $T_\theta:\mathcal{N}(H)_{\theta} \to \mathcal{N}(G)_\theta$, such that $T_0f=T_\theta f$ for all $f \in \mathcal{N}(H)$ and
\begin{align*}
\Vert T_\theta \Vert_{\mathcal{N}(H)_\theta \to \mathcal{N}(G)_\theta } \leq \Vert T_0 \Vert_{\mathcal{N}(H) \to \mathcal{N}(G)}^{1-\theta} \cdot \Vert T_1 \Vert_{L^2(H) \to L^2(G)}^{\theta}. 
\end{align*}
\end{corollary}
For a proof we refer to \cite[Corollary 5.5.4]{bergh1976interpolation}.

\subsection{Gaussian Random Variables}
The main objective is to consider a Gaussian random variable $X$ whose values are functions, or equivalence classes of functions, from $G$ to $\R$, and then condition $X$ on its restriction to a subgroup $H$, writing $Y=X|_H$. This induces an operator $M$ such that $MY=\mathbb{E}(X\vert Y)$. In the main result, Theorem \ref{Theorem:MainResult}, we investigate when this operator is a bounded mapping between $L^2(H)$ and $L^2(G)$. 

Throughout this work, $(\Omega, \mathcal{A}, \mu)$ denotes a probability space, $ E $ and $F$ are separable Banach spaces, and $ E' $ denotes the dual of $E$.
A random variable $ X: \Omega \to \mathbb{R} $ is called a (one-dimensional) Gaussian random variable if there exist $ \mu_X, \sigma_X \in \mathbb{R} $ such that
\begin{align*}
\mathbb{E}\left( e^{i t X} \right) = e^{i t \mu_X - \frac{1}{2} \sigma_X^2 t^2}, \quad \text{for all } t \in \mathbb{R}.
\end{align*}
In this case, we write $ X \sim \mathcal{N}(\mu_X,\sigma_X^2) $. 
For $1 \leq p < \infty$ let $L^p(\mu,E)$ denote the Bochner space, see \cite[Chapter 1]{martingal}.
We note that $ \sigma_X = 0 $ yields a Dirac measure at $ \mu_X $, while $ \sigma_X > 0 $ corresponds to the usual normal distribution.

A random variable $ X: \Omega \to E $ is called Gaussian if for every $ e' \in E' $, the real-valued random variable $ e'(X) :\Omega \to \mathbb{R} $ is Gaussian. We note that separability of $E$ ensures measurability of $X$. Additionally, Fernique’s theorem ensures $X \in L^p(\mu,E) \, \, \text{for all } p \geq 1$, see \cite[Theorem 5.3]{Fernique53}. We call $X$ centered if $\mathbb{E}(X)=0$. From this point on $X$ shall always be centered.

The main setting of this work is $E=L^2(G)$ and $F=L^2(H)$. Other spaces, such as $L^p$ spaces with $1 \leq p \leq \infty$ or spaces of bounded continuous functions, are also natural, but will not be treated in detail here.

Assume now that $G$ is compact and that $X:\Omega \to L^2(G)$ is a Gaussian random variable. Since elements of $L^2(G)$ are equivalence classes, the restriction $X|_H$ is not automatically well-defined if $H$ has Haar measure zero as a subset of $G$. 
We therefore assume that $X$ admits a continuous version, that is, there exists a representative $\widetilde X(\omega)\in X(\omega)$ with $\widetilde X(\omega)\in C(G)$ for $\mu$-almost all $\omega\in\Omega$. We now define
\begin{align*}
    Y(\omega):=\widetilde X(\omega)|_H
\end{align*}
as an element of $L^2(H)$. In this way, the restriction of $X$ to $H$ becomes a well-defined Gaussian random variable $Y:\Omega \to L^2(H)$. The same continuous version also allows us to define the covariance kernel $k_X:G \times G \to \R$ of $X$ pointwise by
\begin{align*}
    k_X(x,y):= \langle \delta_x(\widetilde X),\delta_y(\widetilde X)\rangle_{L^2(\Omega)},
\end{align*}
with $x,y \in G$, where $\delta_x$ denotes the point evaluation at $x$. We call a Gaussian random variable $X$ stationary if $k_X$ is a stationary kernel.  

We denote the RKHS associated with the covariance kernel $k_X$ by $W_X$ and call it the abstract Wiener space of $X$. 
Similarly, if $Y:=\tilde{X}|_H$, then the RKHS associated with $k_Y$ is denoted by $W_Y$. Since $k_Y=k_X|_{H \times H}$, the restriction map induces a canonical bounded operator $L_W:W_X \to W_Y$, by $L_W f= f|_H$. 
Calculating the adjoint of $L_W$, one obtains the mapping $M_W:W_Y \to W_{X}$. The mapping $M_W$ is precisely the minimal-norm extension operator from the restricted RKHS to the original RKHS. In the Gaussian setting, this operator describes the conditional mean. 
More precisely, if $M_W$ extends to a bounded operator $M:L^2(H) \to L^2(G) $ then $MY=\E(X|Y)$, see \cite{WinkleAnalysis}.

If $G$ is not compact, stationary Gaussian random variables need not take values in $L^2(G)$. 
Indeed, for a stationary process the pointwise variance is constant, and the Haar measure of $G$ may be infinite. Thus, an $L^2(G)$-valued formulation generally requires additional structure, for instance the use of weighted spaces $L^2(G,w)$ with a suitable weight $w\in L^1(G)$. Similar issues arise when working with spaces of continuous functions equipped with weighted norms. For this reason, the probabilistic interpretation in terms of $L^2(G)$-valued Gaussian random variables will mainly be considered for compact abelian groups. A standard example is the $d$-dimensional torus $\mathbb T^d$.

Finally, we emphasize that this paper works with real-valued Gaussian random variables. Since harmonic analysis naturally involves complex characters, it may also be interesting to develop an analogous theory for complex-valued Gaussian random variables. For background on complex Gaussian analytic functions, see \cite{hough2009zeros}.

\section{Main Results}\label{sec:results}
The proofs of the main results Theorem \ref{Theorem:MainResult} and Theorem \ref{Theorem:MainCompactResult} can be found in Appendix \ref{sec:5}.
The central object is the minimal-norm extension operator
\begin{align*}
    M_W:\mathcal N(H)\to\mathcal N(G),
\end{align*}
which is naturally defined on the RKHS associated with the restricted kernel. In the probabilistic setting, this operator describes the conditional mean after observing a Gaussian random variable on the subgroup $H$, see \cite{WinkleAnalysis}. However, in applications one often wants to apply the conditional expectation operator to observations that are not given as elements of the RKHS, but rather as elements of a larger space such as $L^2(H)$, because realizations of Gaussian random variables leave the RKHS \cite{driscoll1973reproducing,lukic2001stochastic}. This leads to the question whether $M_W$ admits a bounded extension
\begin{align*}
    M:L^2(H)\to L^2(G).
\end{align*}

This boundedness question is important for approximation. Even if an explicit analytic formula for $M$ is available, it may be difficult to evaluate directly. If $M$ is bounded on $L^2(H)$, then an observation $g\in L^2(H)$ can be approximated by functions $g_n\in\mathcal N(H)$, and the images $M_Wg_n$ converge to $Mg$ in $L^2(G)$. In particular, when the approximants $g_n$ are finite linear combinations of kernel translates on $H$ or linear combinations of characters, meaning that $g_n=\sum_{j=1}^n \hat{g}_n(\xi_j) \cdot\xi_j$ for some $\xi_j \in \hat{H}$, the corresponding extensions $M_Wg_n$ can be computed explicitly by linearity. Thus the $L^2$-boundedness of $M$ provides a practical way to approximate the conditional mean operator. If, in addition, $M$ is bounded from below, meaning that there exist $c,C>0$ such that
\begin{align*}
c\Vert g\Vert_{L^2(H)}
\leq
\Vert Mg \Vert_{L^2(G)}
\leq
C\Vert g\Vert_{L^2(H)}
\qquad
\text{for all } g\in L^2(H),
\end{align*}
and if $g_n$ is a best approximation of $g$ from a finite-dimensional subspace $V_n\subseteq L^2(H)$, then $Mg_n$ is a quasi-best approximation of $Mg$ from $M(V_n)$.

The first theorem gives an explicit Fourier representation of $M_W$ for stationary kernels on locally compact abelian groups and characterizes when this operator extends boundedly from $L^2(H)$ to $L^2(G)$. The second theorem specializes this result to compact groups, where it has a direct interpretation for stationary Gaussian random variables.
\begin{theorem}\label{Theorem:MainResult}
    Let $G$ be a locally compact abelian group and let $H\subseteq G$ be a closed subgroup. Let $k:G\times G \to \R$ be a stationary kernel with spectral density $\hat{k} \in L^1(\hat{G}) \cap L^\infty(\hat{G})$, and assume that the spectral density $\hat{k}_H(\xi_H)$ of the kernel $k_H:H \times H \to \R$ given by $k_H(x,y):=k(x,y)$ for all $x,y \in H$ belongs to $ L^\infty(\hat{H})$. Then the minimal-norm extension operator $M_W: \mathcal{N}(H) \to \mathcal{N}(G)$ is given on the Fourier side by 
    \begin{align*}
        \widehat{M_W g}(\xi_H+\xi^\perp) = \hat{g}(\xi_H) \frac{\hat k(\xi_H+\xi^\perp)}
{\hat k_H(\xi_H)} 
    \end{align*}
    for $g \in \mathcal{N}(H)$, for almost all $\xi_H \in \hat{G}/H^\perp$ and $\xi^\perp \in H^\perp$. On the set where $\hat{k}_H=0$, the quotient is interpreted as zero.  

    Set $A:\hat{G}/H^\perp \to \R$ as
    \begin{align*}
        A(\xi_H):= \frac{
\int_{H^\perp}
|\hat k(\xi_H+\xi^\perp)|^2 \d\xi^\perp
}{
|\hat k_H(\xi_H)|^2
},
\qquad
\xi_H\in \hat G/H^\perp,
    \end{align*}
with the convention $A(\xi_H)=0$ whenever $\hat{k}_H(\xi_H)=0$. If $A \in L^\infty(\hat{G}/H^\perp)$ then the same Fourier formula defines a bounded operator $M:L^2(H) \to L^2(G)$ satisfying $M|_{\mathcal{N}(H)}=M_W$, and 
\begin{align*}
    \Vert M \Vert_{L^2(H) \to L^2(G)}^2 = \esssup_{\xi_H \in \hat{G}/H^\perp} A(\xi_H).
\end{align*}
Moreover, 
\begin{align*}
    \inf_{ \Vert g \Vert_{L^2(H)}=1} \Vert M g \Vert_{L^2(G)}^2 = \essinf_{\xi_H \in \hat{G}/H^\perp} A(\xi_H).
\end{align*}
Finally, for $0 \leq \theta \leq 1$, the operator $M$ restricts to a bounded operator $M_\theta:\mathcal{N}(H)_\theta \to \mathcal{N}(G)_\theta$ and 
\begin{align*}
\Vert M_\theta \Vert_{\mathcal{N}(H)_\theta \to \mathcal{N}(G)_\theta} \leq \esssup_{\xi_H \in \hat{G}/H^\perp} A(\xi_H)^{\theta/2}.
\end{align*}
\end{theorem}

The next theorem focuses on compact $G$, allowing us to consider stationary Gaussian random variables $X$.

\begin{theorem}\label{Theorem:MainCompactResult}
    Let $G$ be a compact abelian group and let $H\subseteq G$ be a closed subgroup. Let
$X:\Omega\to L^2(G)$ be a centered stationary Gaussian random variable with continuous realizations, and let $k_X$ be its covariance kernel with spectral density $\hat k_X$. Define $Y:=\tilde{X}|_H$. Then the Fourier formula 
\begin{align*}
    \widehat{Mg}(\xi)= \hat{g}(\xi_H) \frac{\hat{k}_X(\xi)}{\hat{k}_Y(\xi_H)}, \qquad \xi \in \hat{G}, \quad \xi_H =\xi+H^\perp,
\end{align*}
with the quotient interpreted as zero whenever $\hat{k}_Y(\xi_H)=0$, defines a bounded operator $M:L^2(H) \to L^2(G)$. Equivalently 
\begin{align*}
    Mg
=
\sum_{\xi\in\widehat G}
\hat g(\xi_H)
\frac{\hat k_X(\xi)}
{\hat k_{Y}(\xi_H)}
\xi,
\qquad
\xi_H:=\xi+H^\perp.
\end{align*} 
Moreover, 
\begin{align*}
  \Vert  Mg\Vert_{L^2(G)}^2
\leq
\left[
\sup_{\xi_H\in\hat G/H^\perp}
\frac{
\sum_{\xi^\perp\in H^\perp}
|\hat k_X(\xi_H+\xi^\perp)|^2
}{
\left|
\sum_{\xi^\perp\in H^\perp}
\hat k_X(\xi_H+\xi^\perp)
\right|^2
}
\right]
\Vert g\Vert_{L^2(H)}^2
\leq
\Vert g\Vert_{L^2(H)}^2 .
\end{align*}
If $W_Y$ is dense in $L^2(H)$, then $M$ is the unique bounded extension of the minimal-norm extension operator $M_W:W_Y \to W_X$. In this case 
\begin{align*}
    MY= \E(X \vert Y).
\end{align*} 
Finally, since $W_X=\mathcal{N}(G)$ and $W_Y=\mathcal{N}(H)$, it follows that, for $0 \leq \theta \leq 1$, the operator $M$ restricts to a bounded operator $M_\theta:\mathcal{N}(H)_\theta \to \mathcal{N}(G)_\theta$ and
\begin{align*}
    \Vert M_\theta \Vert_{\mathcal{N}(H)_\theta \to \mathcal{N}(G)_\theta} \leq 1.
\end{align*}
\end{theorem}

\color{black}
\section{Examples}\label{sec:examples}\label{sec:4}

We begin this section with a counterexample to the contraction statement in \cite[Theorem 2.4]{lagatta2013continuous}. Afterwards, we discuss two examples in which the abstract results can be written in elementary Fourier-analytic terms. The first example is cardinal interpolation, corresponding to $G=\mathbb R$ and the subgroup $H_h=h\mathbb Z$. The second example is set on the two-dimensional torus $G=\mathbb T^2$, where we consider the one-dimensional closed subgroup $H_n:=\{(a,na)\in\mathbb T^2\mid a\in\mathbb R/\mathbb Z\}$.

\subsection{Counterexample to Theorem 2.4 in \texorpdfstring{\cite{lagatta2013continuous}}{Lagatta et al. (2013)}}
We first recall the consequence of \cite[Theorem 2.4 and Corollary 2.5]{lagatta2013continuous} that will be relevant for the counterexample below. Let $X:\Omega\to C([0,2])$ be a centered Gaussian random variable with covariance kernel
\begin{align*}
k_X(t,s)=\kappa(t-s),
\qquad t,s\in[0,2],
\end{align*}
where $\kappa:[-2,2]\to\mathbb R$. For an interval $I\subseteq[0,2]$, set $Y:=X|_I$. The cited result would imply the existence of an operator
$M:C(I)\to C([0,2])$
such that
\begin{align*}
MY=\mathbb E(X\mid Y) \qquad \textup{and} \qquad \Vert M\Vert_{C(I)\to C([0,2])}\leq 1,
\end{align*}
where both spaces are equipped with the supremum norm. We now show that this contraction estimate cannot hold in this generality.

For the counterexample, consider the Gaussian covariance function $\kappa:[-2,2] \to \R$ with $\kappa(t)=\e^{-t^2}$. The corresponding Gaussian process has analytic sample paths, see \cite{belyaev1959analytic}. Hence observing the process on any non-empty interval determines it everywhere by analytic continuation. In the RKHS formulation, the minimal-norm extension operator therefore coincides with the unique analytic continuation operator, see also \cite{lichtenberger1974note}.

Let $f:[0,2] \to \R$ be given by $f(x)=x^2 \cdot \e^{-x^2}$. Then $f \in W_X$. Moreover, $f$ attains its maximum in $[0,2]$ at $x=1$, hence $\Vert f\Vert_{C([0,2])}=f(1)=\frac{1}{\mathrm e}$. Now choose $I=[0,1/2]$. Since $f$ is increasing on $I$, we have $\Vert f \Vert_{C(I)} = f(1/2) =1/(4\e^{1/4})$. Since $f$ is analytic and $M_W$ is the analytic continuation operator, we have $M_W (f|_I)=f$. 
Thus, any operator $M:C(I)\to C([0,2])$ that extends this conditional expectation operator must satisfy 
\begin{align*}
\Vert M \Vert_{C(I) \to C([0,2])} \geq \frac{\Vert f\Vert_{C([0,2])}}{\Vert f|_I\Vert_{C(I)}} = 4 \mathrm{e}^{-3/4} >1.
\end{align*}
Consequently, the contraction estimate claimed in \cite[Theorem~2.4 and Corollary~2.5]{lagatta2013continuous} cannot hold in this generality. The error in the proof of \cite[Theorem~2.4]{lagatta2013continuous} lies in the incorrect assumption that the linear span of the point-evaluation functionals is norm-dense in the dual space $C(I)'$. In fact, this span is only weak-$*$-dense in $C(I)'$, see \cite[Page 114]{ReedSimonI}, which is insufficient to establish the conclusion of \cite[Theorem~2.4]{lagatta2013continuous}.

\subsection{Cardinal Interpolation}

In the context of cardinal interpolation, we refer to \cite{hangelbroek2012cardinal} for Gaussian kernels and to \cite{bejancu2022uniformly} for Matérn-type kernels.

For simplicity, we restrict ourselves to the one-dimensional case. Thus, let $G=\R$ and $H_h=h \Z$ with $h>0$. We use the Fourier transform convention
\begin{align*}
    \hat{f}(t)= \int_{\R} f (x) \e^{-2\pi \i xt } \d x, \quad \textup{and} \quad f(x)=\int_{\R} \hat{f}(t) \e^{2\pi \i x t} \d t
\end{align*}
where the inversion formula is understood in the $L^2(\mathbb R)$-sense for $f\in L^2(\mathbb R)$.

On $h\mathbb Z$, we use the normalized counting measure $h\sum_{m\in\mathbb Z}\delta_{hm}$. Thus, for $g\in\ell^2(h\mathbb Z)$, 
\begin{align*}
    \Vert g\Vert_{\ell^2(h\mathbb Z)}^2
=
h\sum_{m\in\mathbb Z}|g(hm)|^2.
\end{align*}
With this normalization, the Fourier transform on $h\mathbb Z$ is given by
\begin{align*}
    \hat{g}(t)=h \sum_{m \in \Z} g(hm) \e^{-2\pi \i hmt},\qquad
t\in[0,1/h), 
\end{align*}
and the inverse formula is 
\begin{align*}
    g(hm)= \int_0^{1/h} \hat{g}(t) \e^{2\pi \i hm t } \d t, \qquad m \in \Z.
\end{align*}
Again, for $g\in\ell^2(h\mathbb Z)$, this is understood in the corresponding $L^2$-Fourier sense.
Given a sufficiently smooth function $f \in C_c(\R)$ and set $g:=f|_{H_h}$ we obtain for $m\in \Z$ that $g(hm)=f(hm)$. The Fourier coefficients of $g$ are thus given by 
\begin{align*}
    \hat{g}(t)=  \sum_{j \in \Z} \hat{f}\left( t+ \frac{j}{h} \right), \quad \forall t \in [0,1/h).
\end{align*}

Before applying Theorem \ref{Theorem:MainResult}, we need to assume that the Condition \eqref{eq:ConditionPrelim} holds meaning 
\begin{align}\label{eq:cardinalCondition}
    \sup_{t \in [0,1/h)} \sum_{j \in \Z} \hat{k}\left( t + \frac{j}{h} \right) < \infty.
\end{align}
Here the annihilator of $h\mathbb Z$ is $\frac1h\mathbb Z$, and $\hat{\mathbb R}/(h\mathbb Z)^\perp$ is identified with $[0,1/h)$.

We denote the operator of the minimal-norm solution of \eqref{eq:condition} by $M_h$ since it depends on $h$.
By Theorem \ref{Theorem:MainResult}, the Fourier transform of $M_h g$ is then given by 
\begin{align*}
    \widehat{M_h g}\left(t+\frac{j}{h} \right) =  \hat{g}(t) \frac{\hat{k}\left(t+\frac{j}{h} \right)}{\sum_{l \in \Z} \hat{k}\left( t+ \frac{l}{h} \right)} \qquad t\in[0,1/h),\ j\in\mathbb Z .
\end{align*}
Define 
\begin{align*}
    A_h(t):=\frac{\sum_{j \in \Z} \left( \hat{k}\left(t+\frac{j}{h} \right) \right)^2}{\left(\sum_{j \in \Z} \hat{k}\left( t+ \frac{j}{h} \right) \right)^2}, \qquad t\in[0,1/h).
\end{align*}
Since $\hat{k}\geq0$, we have  $A_h(t) \leq 1$ for all $ t \in [0,1/h)$. Consequently, Theorem \ref{Theorem:MainResult} gives $\Vert M_h \Vert_{\ell^2(h\Z ) \to L^2(\R)} \leq 1$. 

To investigate whether $M_h$ is also bounded from below, we consider two examples. First, let $\hat{k}(t)=\e^{-t^2}$. Then condition \eqref{eq:cardinalCondition} is satisfied by Lemma \ref{lem:Gaussianupperandlowerbound}. In this case, 
\begin{align*}
    A_h(t)=\frac{\sum_{j \in \Z} \e^{-2\left(t+\frac{j}{h}\right)^2}}{\left(\sum_{j \in \Z} \e^{-\left(t+\frac{j}{h} \right)^2} \right)^2}, \qquad t\in[0,1/h).
\end{align*}
By Lemma \ref{lem:Gaussianupperandlowerbound}, we have 
\begin{align*}
    \frac{\e^{-1/(2h^2)}}{(1+\sqrt{\pi}h)^2} \leq A_h(t) \leq 1, \qquad  t \in[0,1/h).
\end{align*}
Therefore, Theorem \ref{Theorem:MainResult} yields for all $g \in \ell^2(h \Z)$,
\begin{align*}
    \frac{\e^{-1/(2h^2)}}{(1+\sqrt{\pi}h)^2}    \Vert g \Vert_{\ell^2(h\Z)}^2 \leq \Vert M_h g \Vert_{L^2(\R)}^2 \leq \Vert g \Vert_{\ell^2(h \Z )}^2.
\end{align*}

As a second example, consider the Sobolev-type spectral density $\hat{k}(t)=(1+t^2)^{-\tau}$ for $\tau>1/2$. Then condition \eqref{eq:cardinalCondition} is satisfied by Lemma \ref{lem:Sobolevupperandlowerbound}. In this case, 
\begin{align*}
    A_h(t)= \frac{\sum_{j \in \Z} \left(1+\left(t+\frac{j}{h}\right)^2\right)^{-2\tau}}{\left(\sum_{j \in \Z} \left( 1+\left(t+\frac{j}{h}\right)^2\right)^{-\tau} \right)^2}, \qquad t\in[0,1/h).
\end{align*}
By Lemma \ref{lem:Sobolevupperandlowerbound}, there exists a constant $C_\tau>0$ such that
\begin{align*}
\frac{\left( 1+ \frac{1}{2h^2} \right)^{-2\tau}}{(2+2C_\tau h)^2}    \leq A_h(t) \leq 1, \qquad  t \in [0,1/h).
\end{align*} 
Therefore, Theorem \ref{Theorem:MainResult} yields, for all $g \in \ell^2(h \Z)$,
\begin{align*}
   \frac{\left( 1+ \frac{1}{2h^2} \right)^{-2\tau}}{(2+2C_\tau h)^2}  \Vert g \Vert_{\ell^2(h\Z)}^2 \leq \Vert M_h g \Vert_{L^2(\R)}^2 \leq \Vert g \Vert_{\ell^2(h \Z )}^2.
\end{align*}
\subsection{Conditioning on a Subgroup of the Torus}
Let $G=\mathbb{T}^2$ be the two-dimensional torus, and, for $n \in \N$, we define $H_n:=\{ (a,na) \in \T^2 \, \vert \, a \in \R / \Z \}$. 
We use the normalized Haar measure on $H_n$, identified with the Lebesgue measure on $\mathbb R/\mathbb Z$, that is, for $f \in L^2(H_n)$, $\int_{H_n} f \, \d \mu =\int_0^1 f(a,na)  \d a$. 
For $f\in L^2(\mathbb T^2)$, we use the Fourier convention $f(x,y)=\sum_{(m,l) \in \Z^2} \hat{f}(m,l) \e^{2 \pi \i (mx+ly)}$ where $\hat{f}(m,l)= \int_0^1 \int_0^1 f(x,y) \e^{-2\pi \i (mx+ly)} \, \d x  \d y$. If $g\in L^2(H_n)$, we identify $g$ with the function $g_n: \R/\Z \to \R$ and $g_n(a):=g(a,na)$. Then $g_n(a)=\sum_{l \in \Z} \hat{g}_n(l) \e^{2\pi \i l a}$, with $\hat{g}_n(l)=\int_0^1 g_n(a) \e^{-2\pi \i la } \d a$. 

Now, let $f \in C(\T^2)$ and let $g=f|_{H_n}$. Then $g_n(a)=f(a,na)$. Comparing Fourier coefficients gives 
\begin{align*}
    \hat{g}_n(l)=\sum_{(m,r) \in \Z^2: m+nr=l} \hat{f}(m,r)= \sum_{q \in \Z} \hat{f}(l-nq,q).
\end{align*}
Moreover, $H_n^\perp=\{ (m,r) \in \Z^2 \, \vert \, m+nr=0\}$. Let $X$ be a stationary Gaussian random variable on $\T^2$ with covariance kernel $k$ whose spectral density is $\hat{k}(m,r)=(1+(m^2+r^2))^{-\tau}$ for $\tau>2$, see \cite{Slava2020}. Note that we require $\tau>2$ so that $X$ has continuous paths, see \cite[Theorem 4.1]{steinwart2024paths}.
By Theorem \ref{Theorem:MainCompactResult}, the conditioning operator $M_n:L^2(H_n) \to L^2(\T^2)$ is given on the Fourier side by
\begin{align*}
    \widehat{M_n g_n}(l-nq,q)&= \hat{g}_n(l) \frac{\hat{k}(l-nq,q)}{\sum_{r \in \Z} \hat{k}(l-nr,r)}, \qquad  (l,q) \in \Z^2.
\end{align*}
Equivalently, for $(t,s) \in \Z^2$, writing $l=t+ns$, we have 
\begin{align*}
        \widehat{M_n g_n}(t,s)&= \hat{g}_n(t+ns) \frac{\hat{k}(t,s)}{\sum_{r \in \Z} \hat{k}(t+ns-nr,r)}.
\end{align*}
Thus,
\begin{align*}
    M_n g_n (x,y)= \sum_{(t,s) \in \Z^2} \hat{g}_n(t+ns) \frac{(1+(t^2+s^2))^{-\tau}}{\sum_{r \in \Z} (1+(t+ns-nr)^2+r^2)^{-\tau}} \mathrm{e}^{2 \pi \i (tx+sy)}.
\end{align*}
Again, Theorem \ref{Theorem:MainCompactResult} gives $\Vert M_n \Vert_{L^2(H_n) \to L^2(\T^2)} \leq 1$ and $\Vert M_n \Vert_{\mathcal{N}(H_n)_\theta \to \mathcal{N}(G)_\theta} \leq 1$. 

Finally, let $l_0 \in \Z $ and consider the function $f(a,b)=\e^{2\pi \i l_0 a}$. Then $g(a,na)=\e^{2\pi \i l_0 a }$ and 
\begin{align*}
    \hat{g}_n(l)=\begin{cases}
1, \, \, l=l_0 \\
0, \, \, l\neq l_0.
    \end{cases}
    \end{align*}
Hence,
    \begin{align*}
        M_n g_n(x,y)= \sum_{q \in \Z} \frac{(1+(l_0-nq)^2+q^2)^{-\tau}}{\sum_{r \in \Z} (1+(l_0-nr)^2+r^2)^{-\tau}} \e^{2\pi \i ((l_0-nq)x+qy)}.
    \end{align*}

 \color{black}

\paragraph*{Acknowledgements.}{I would like to thank David Ginsbourger for bringing \cite{lagatta2013continuous} to my attention and for the helpful suggestion to consider interpolation spaces. This work was supported by the Research Council of Finland project 368086.}

\bibliographystyle{plain}
\bibliography{daniel_refs}

\newpage

\begin{appendices}
\section{Auxiliary Results}\label{sec:appendix}\label{sec:6}

Before proving our main results, we establish several necessary technical results that will also be used in the examples.

\begin{lemma}\label{lem:Scalarproduct} 
    Let $k$ be a stationary kernel. Then the norm on $\mathcal{N}(G)$ is given by 
    \begin{align*}
        \Vert f \Vert_{\mathcal{N}(G)}^2 = \int_{\hat{G}} \frac{\vert \hat{f}(\xi)\vert^2}{\hat{k}(\xi)} \d \xi.
    \end{align*}
\end{lemma}
\begin{proof}
Consider the Hilbert space
\begin{align*}
    L^2(\hat{G},\hat{k}(\xi) \d \xi) :=\left\{ \varphi:\hat{G} \to \C \, \middle\vert \, \, \varphi \textup{ is measurable and } \Vert \varphi\Vert_{L^2(\hat k)}^2:= \int_{\hat{G}} {\vert \varphi (\xi) \vert^2}{\hat{k}(\xi)} \d \xi < \infty \right\}.
\end{align*} 
By the standard $L^2$ construction, this is a Hilbert space. Moreover, we define $\Phi:G \to L^2(\hat{G},\hat{k}(\xi) \d \xi) $ by  $\Phi(x) = \xi(x)$.
Then $k(x,y)= \langle \Phi(x),\Phi(y)\rangle_{L^2(\hat k )}$. 
Applying \cite[Theorem 4.21]{SteinwartSVM} shows that the RKHS norm is given by 
\begin{align*}
    \Vert f \Vert_{\mathcal{N}(G)}^2 = \inf \{ \Vert \varphi \Vert_{L^2(\hat{k})}^2 \, \vert \, {\varphi \in L^2(\hat{G}, \hat k \d \xi ) } \wedge f(x)=\langle \varphi, \Phi(x) \rangle_{L^2(\hat{k})} \, \forall x \in G\}. 
\end{align*}

    We define $T:L^2(\hat{G},\hat{k}(\xi) \d \xi) \to C_b(G)$ by 
    \begin{align*}
       ( T \varphi) (x)=\int_{\hat{G}} \varphi(\xi) \xi(x) \hat{k}(\xi) \d \xi.
    \end{align*}
    This is well-defined by Cauchy--Schwarz. Indeed, since $\hat{k} \in L^1(\hat{G})$, we have for all $\varphi \in L^2(\hat{G},\hat{k}(\xi) \d \xi)$ that
    \begin{align*}
    \left(        \int_{\hat{G}} \vert \varphi (\xi) \vert \hat{k}(\xi) \d \xi \right)^2 
= \left(        \int_{\hat{G}} \left( \vert \varphi (\xi) \vert  \sqrt{\hat{k}(\xi)}\right) \cdot \left( \sqrt{\hat{k}(\xi)}\right)\d \xi \right)^2 
\leq \int_{\hat{G}} \vert \varphi (\xi) \vert^2 \hat{k}(\xi) \d \xi \cdot \int_{\hat{G}} \hat{k}(\xi) \d \xi.
    \end{align*}
Hence, $T\varphi$ is the inverse Fourier transform of an $L^1(\hat{G})$ function. Additionally, by \cite[Theorem 4.21]{SteinwartSVM} $T\varphi \in \mathcal{N}(G)$.

Now, let $f = T \varphi$. By uniqueness of the Fourier transform, see \cite[Theorem 4.33]{Folland}, we have 
\begin{align*}
    \hat{f}(\xi) = \varphi(\xi) \hat{k}(\xi).
\end{align*}
Hence, on $\{\hat{k}>0\}$ we have $\varphi(\xi)=\hat{f}(\xi)/\hat{k}(\xi)$. Moreover, on $\{\hat{k} = 0 \}$ we have $\hat{f}(\xi)=0$. Thus,
\begin{align*}
    \Vert f \Vert_{\mathcal{N}(G)}^2= \Vert \varphi \Vert_{L^2(\hat k )}^2 = \int_{\hat{G}} \vert \varphi(\xi)\vert^2 \hat{k}(\xi) \d \xi = \int_{\{\hat{k}>0\}} \frac{\vert \hat{f} (\xi) \vert^2}{\hat{k}(\xi)} \d \xi
\end{align*}
and the assertion follows.
\end{proof}

\begin{lemma}\label{lem:RKHSProperties}
    Let $k$ be a stationary kernel. Then $ \mathcal{N}(G) \subseteq C_b(G) \cap L^2(G)$, and for every $f \in \mathcal{N}(G)$ one has $\hat{f} \in L^1(\hat{G}) \cap L^2(\hat{G}) $. 
\end{lemma}
\begin{proof}
Since \(f\in\mathcal N(G)\), Lemma \ref{lem:Scalarproduct} gives
\[
\|f\|_{\mathcal N(G)}^2
=
\int_{\hat G}
\frac{|\hat f(\xi)|^2}{\hat k(\xi)}\,d\xi
<\infty.
\]
Using \(\hat k\in L^\infty(\hat G)\), we obtain
\[
\|\hat f\|_{L^2(\hat G)}^2
=
\int_{\hat G}|\hat f(\xi)|^2\,d\xi
=
\int_{\hat G}
\frac{|\hat f(\xi)|^2}{\hat k(\xi)}
\hat k(\xi)\,d\xi
\le
\|\hat k\|_{L^\infty(\hat G)}
\|f\|_{\mathcal N(G)}^2.
\]
Hence \(\hat f\in L^2(\hat G)\). By Plancherel,
\[
\|f\|_{L^2(G)}
=
\|\hat f\|_{L^2(\hat G)}
\le
\|\hat k\|_{L^\infty(\hat G)}^{1/2}
\|f\|_{\mathcal N(G)}.
\]
Thus \(\mathcal N(G)\subseteq L^2(G)\) continuously.

Furthermore, since \(\hat k\in L^1(\hat G)\), Cauchy--Schwarz gives
\[
\|\hat f\|_{L^1(\hat G)}
=
\int_{\hat G}
\frac{|\hat f(\xi)|}{\sqrt{\hat k(\xi)}}
\sqrt{\hat k(\xi)}\,d\xi
\le
\|f\|_{\mathcal N(G)}
\|\hat k\|_{L^1(\hat G)}^{1/2}.
\]
Therefore \(\hat f\in L^1(\hat G)\cap L^2(\hat G)\).

Finally, because \(\hat f\in L^1(\hat G)\), Fourier inversion yields
\[
f(x)=\int_{\hat G}\hat f(\xi)\xi(x)\,d\xi,
\]
so \(f\) has a continuous bounded representative. Hence
\[
\mathcal N(G)\subseteq C_b(G)\cap L^2(G).
\]
\end{proof}

\begin{lemma}\label{lem:INL2}
Let $H \subseteq G$ be a closed subgroup. Given a stationary kernel $k$ on $G$, the spectral density $\hat{k}_H:\hat{H} \to \R$ of the restricted kernel is given by 
\begin{align*}
    \hat{k}_H(\xi_H)= \int_{H^\perp} \hat{k}(\xi_H+\xi^\perp) \d \xi^\perp.
\end{align*}
If 
\begin{align}\label{eq:L2Condition}
    \sup_{\xi_H \in \hat{G}/H^\perp} \int_{H^\perp} \hat{k}(\xi_H + \xi^\perp) \d \xi^\perp < \infty,
\end{align}
additionally holds, then $\mathcal{N}(H) \subseteq L^2(H)$. 

If $G$ is compact, \eqref{eq:L2Condition} is satisfied.
\end{lemma}
\begin{proof}
    Define $k_H:H \times H \to \R$ by $k_H(t,s)=k(t,s)$ for all $t,s \in H$. This implies 
    \begin{align*}
        k_H(t,s)=\int_{\hat{G}} \hat{k}(\xi)\xi(t-s) \d \xi.
    \end{align*}
    By Lemma \ref{lem:Fubini} we have 
    \begin{align*}
        \int_{\hat{G}} \hat{k}(\xi)\xi(t-s) \d \xi 
        &= \int_{\hat{G}/H^\perp} \int_{H^\perp} \hat{k}(\xi_H+\xi^\perp) (\xi_H+\xi^\perp)(t-s) \d \xi^\perp \d \xi_H \\
        &= \int_{\hat{G}/H^\perp} \int_{H^\perp} \hat{k}(\xi_H+\xi^\perp) \xi_H(t-s) \d \xi^\perp \d \xi_H \\
        &= \int_{\hat{G}/H^\perp} \left[ \int_{H^\perp} \hat{k}(\xi_H+\xi^\perp) \d \xi^\perp\right] \xi_H(t-s) \d \xi_H,
    \end{align*}
    note that we used $\xi^\perp(t-s) = 1 $ for all $t,s \in H$ and $\xi^\perp \in H^\perp$. Thus the spectral density of $k_H$ is given by 
    \begin{align*}
        \hat{k}_H(\xi_H)=\int_{H^\perp} \hat{k}(\xi_H+\xi^\perp) \d \xi^\perp, \quad \forall \xi_H \in \hat{G}/H^\perp \cong \hat{H}.
    \end{align*}
    By stationarity, $\hat{k} \in L^1(\hat{G}) \cap L^\infty(\hat{G})$, and hence $\hat{k}_H \in L^1(\hat{H})$. Assumption \eqref{eq:L2Condition} additionally gives $\hat{k}_H \in L^\infty(\hat{H})$. Lemma \ref{lem:RKHSProperties} therefore implies that $\mathcal{N}(H) \subseteq L^2(H)$.  

    If $G$ is compact we then have 
    \begin{align*}
        \hat{k}_H(\xi_H)= \sum_{\xi^\perp \in H^\perp} \hat{k}(\xi_H+\xi^\perp) \leq \sum_{\xi \in \hat{G}} \hat{k}(\xi) < \infty.
    \end{align*}
    In the last step, we used $\hat{k} \in L^1(\hat{G})$.
\end{proof}

\begin{lemma}\label{lem:Gaussianupperandlowerbound}
Let $h>0$ and $A_h:[0,1/h) \to \R$ be given by  
\begin{align*}
    A_h(t)=\frac{\sum_{j \in \Z} \e^{-2\left(t+\frac{j}{h}\right)^2}}{\left(\sum_{j \in \Z} \e^{-\left(t+\frac{j}{h} \right)^2} \right)^2}
\end{align*}
We then have that 
\begin{align*}
    \frac{\e^{-1/(2h^2)}}{(1+\sqrt{\pi}h)^2} \leq A_h(t) \leq 1, \quad \forall t \in[0,1/h).
\end{align*}
Additionally, we have that
\begin{align*}
    \sup_{t \in [0,1/h)} \sum_{j \in \Z} \e^{-\left(t+\frac{j}{h} \right)^2} < \infty.
\end{align*}
\end{lemma}
\begin{proof}
    The upper bound follows directly by $\sum_{j \in \Z} a_j^2 \leq \left( \sum_{j \in \Z} a_j\right)^2$ for any positive sequence $(a_j) \in \ell^1(\Z)$. 

    For the lower bound we first estimate the numerator by 
    \begin{align*}
        \sum_{j \in \Z} \e^{-2\left(t+\frac{j}{h}\right)^2} \geq \e^{-1/(2h^2)}. 
    \end{align*}
    The denominator is estimated by 
    \begin{align*}
        \sum_{j \in \Z} \e^{-\left(t+\frac{j}{h} \right)^2} \leq 1+2 \sum_{j=1}^\infty \e^{-j^2/h^2} \leq 1 + 2 \int_0^\infty \e^{-x^2/h^2} \d x = 1+ \sqrt{\pi} h.
    \end{align*}
Thus the assertion follows.
\end{proof}

\begin{lemma}\label{lem:Sobolevupperandlowerbound}
Let $h>0$, $\tau>1/2$, and $A_h:[0,1/h) \to \R$ be given by
\begin{align*}
    A_h(t)= \frac{\sum_{j \in \Z} \left(1+\left(t+\frac{j}{h}\right)^2\right)^{-2\tau}}{\left(\sum_{j \in \Z} \left( 1+\left(t+\frac{j}{h}\right)^2\right)^{-\tau} \right)^2}
\end{align*}    
Then there exists a constant $C_\tau>0$, independent of $h$ but possibly depending on $\tau$, such that
\begin{align*}
\frac{\left( 1+ \frac{1}{2h^2} \right)^{-2\tau}}{(2+2C_\tau h)^2}    \leq A_h(t) \leq 1, \quad \forall t \in [0,1/h).
\end{align*}
Additionally, we have that
\begin{align*}
    \sup_{t \in [0,1/h)} \sum_{j \in \Z} \left( 1+\left(t+\frac{j}{h}\right)^2\right)^{-\tau} < \infty.
\end{align*}

\end{lemma}
\begin{proof}
       The upper bound follows directly by $\sum_{j \in \Z} a_j^2 \leq \left( \sum_{j \in \Z} a_j\right)^2$ for any positive sequence $(a_j) \in \ell^1(\Z)$. 
           For the lower bound we first estimate the numerator by 
    \begin{align*}
      \sum_{j \in \Z} \left(1+\left(t+\frac{j}{h}\right)^2\right)^{-2\tau} \geq \left( 1+ \frac{1}{2h^2} \right)^{-2\tau}. 
    \end{align*} 
        The denominator is estimated by 
    \begin{align*}
     \sum_{j \in \Z} \left( 1+\left(t+\frac{j}{h}\right)^2\right)^{-\tau} \leq 2 + 2h \int_0^\infty (1+x^2)^{-\tau} \d x \leq 2(1+C_\tau h)
    \end{align*}
    for some constant $C_\tau >0$. Thus the assertion follows.
\end{proof}

\begin{lemma}\label{lem:minimizationproblem}
Let $G$ be a locally compact abelian group and let $H\subseteq G$ be a closed subgroup. Let $k:G \times G \to \R$ be a stationary kernel with spectral density $\hat{k}:\hat{G} \to \R$.
Define
\begin{align*}
    \hat k_H(\xi_H)
:=
\int_{H^\perp}
\hat k(\xi_H+\xi^\perp) \d\xi^\perp,
\qquad
\xi_H\in \hat G/H^\perp\simeq \widehat H.
\end{align*}
Let $g\in\mathcal N(H)$. Then the unique solution of the minimal-norm extension problem
\begin{align*}
    \min
\left\{
\Vert f\Vert_{\mathcal N(G)}^2
=
\int_{\hat G}
\frac{|\hat f(\xi)|^2}{\hat k(\xi)} \d\xi
 \, \middle| \, 
f\in\mathcal N(G),\ f|_H=g \right\}
\end{align*}
is characterized on the Fourier side by
\begin{align*}
    \hat f(\xi_H+\xi^\perp)
=
\hat g(\xi_H)
\frac{\hat k(\xi_H+\xi^\perp)}
{\hat k_H(\xi_H)}
\end{align*}
for almost all $\xi_H\in \hat G/H^\perp$ and $\xi^\perp\in H^\perp$. On the set where $\hat k_H=0$, the quotient is interpreted as zero.
\end{lemma}
\begin{proof}
First observe that the restriction condition can be expressed on the Fourier side. If $f\in\mathcal N(G)$, then, for $x\in H$, the Fourier inversion formula from Theorem \ref{theorem:Fourierinversion} and the Haar-measure decomposition from Lemma \ref{lem:Fubini} give
\begin{align*}
    f(x) 
    = \int_{\hat{G}}    \hat{f}(\xi) \xi(x) \d \xi 
 = \int_{\hat{G}/H^\perp} \left[\int_{H^\perp} \hat{f}(\xi_H + \xi^\perp) (\xi_H+\xi^\perp)(x) \d \xi^\perp \right] \d \xi_H.
\end{align*}
Since $\xi^\perp(x)=1$ for all $x\in H$, this becomes
\begin{align*}
f(x) = \int_{\hat{G}/H^\perp} \left[ \int_{H^\perp} \hat{f}(\xi_H+\xi^\perp) \d \xi^\perp \right] \xi_H(x) \d \xi_H.
\end{align*}
Identifying $\hat H$ with $\hat G/H^\perp$, as in \cite[Theorem 4.39]{Folland}, and using uniqueness of the Fourier transform, we obtain
\begin{align*}
    \hat g(\xi_H)
=
\int_{H^\perp}
\hat f(\xi_H+\xi^\perp) \d\xi^\perp
\end{align*}
for almost all $\xi_H \in \hat{G}/H^\perp$.

By Lemma \ref{lem:Fubini}, the norm in $\mathcal N(G)$ decomposes as 
\begin{align*}
    \Vert f\Vert_{\mathcal N(G)}^2
=
\int_{\hat G}
\frac{|\widehat f(\xi)|^2}{\widehat k(\xi)} \d\xi \
=
\int_{\hat G/H^\perp}
\int_{H^\perp}
\frac{
|\widehat f(\xi_H+\xi^\perp)|^2
}{
\widehat k(\xi_H+\xi^\perp)
}
\d\xi^\perp \d\xi_H .
\end{align*}
Thus the minimization problem separates into independent fiberwise minimization problems. 

Fix $\xi_H\in\hat G/H^\perp$ such that $0<\hat k_H(\xi_H)<\infty$, and set 
\begin{align*}
    K_{\xi_H}(\xi^\perp)
:=
\hat k(\xi_H+\xi^\perp).
\end{align*}
Consider the Hilbert space
\begin{align*}
    H_{\xi_H}
:=
L^2\left(
H^\perp,
K_{\xi_H}(\xi^\perp)^{-1}\d\xi^\perp
\right).
\end{align*}
Consider the functional $L:H_{\xi_H} \to \C$ given by $L(F) := \int_{H^\perp} F(\xi^\perp) \d \xi^\perp$.
The constraint on the fiber is given by 
\begin{align*}
    L(F)=\hat g(\xi_H).
\end{align*}
The functional $L:H_{\xi_H} \to \C$ is bounded, since 
\begin{align*}
    \vert L(F)|
\leq
\left(
\int_{H^\perp}
\frac{|F(\xi^\perp)|^2}{K_{\xi_H}(\xi^\perp)}
 \d\xi^\perp
\right)^{1/2}
\left(
\int_{H^\perp}
K_{\xi_H}(\xi^\perp) \d\xi^\perp
\right)^{1/2}.
\end{align*}
The Riesz representer $\psi_{\xi_H}$ of $L$ is 
\begin{align*}
    \psi_{\xi_H}(\xi^\perp)=K_{\xi_H}(\xi^\perp),
\end{align*} 
and 
\begin{align*}
    \Vert \psi_{\xi_H}\Vert_{H{\xi_H}}^2
=
\int_{H^\perp}
K_{\xi_H}(\xi^\perp)\d\xi^\perp
=
\hat k_H(\xi_H).
\end{align*}
The minimal-norm element satisfying $L(F)=\widehat g(\xi_H)$ is therefore 
\begin{align*}
  F(\xi^\perp)
=
\hat g(\xi_H)
\frac{K_{\xi_H}(\xi^\perp)}
{\hat k_H(\xi_H)}
=
\hat g(\xi_H)
\frac{\hat k(\xi_H+\xi^\perp)}
{\hat k_H(\xi_H)} .
\end{align*}
This gives the asserted formula.

If $\hat k_H(\xi_H)=0$, then $K_{\xi_H}=0$ almost everywhere on $H^\perp$. Since $g\in\mathcal N(H)$, its Fourier transform satisfies $\hat g=0$ almost everywhere on $\{\hat k_H=0\}$. On this set we define the fiberwise minimizer to be zero. Hence the formula holds with the stated convention.
\end{proof}
\end{appendices}
\begin{appendices}
\section{Proofs of the Main Results}\label{sec:proofs}\label{sec:5}
We are now in a position to prove our main results.

\begin{proof}[Proof of Theorem \ref{Theorem:MainResult}]
    By Lemma \ref{lem:minimizationproblem}, the minimal-norm extension $M_Wg$ of $g\in\mathcal N(H)$ is characterized on the Fourier side by 
    \begin{align*}
        \widehat{M_W g}(\xi_H+\xi^\perp)
=
\hat g(\xi_H)
\frac{\hat k(\xi_H+\xi^\perp)}
{\hat k_H(\xi_H)}.
    \end{align*}
Using Plancherel's theorem \ref{theorem:Plancherel} and the Haar-measure decomposition from Lemma \ref{lem:Fubini} we obtain 
\begin{align*}
    \Vert M_Wg\Vert_{L^2(G)}^2
=\int_{\hat G} |\widehat{M_Wg}(\xi)|^2 \d\xi 
&=\int_{\hat G/H^\perp} \int_{H^\perp}|\widehat g(\xi_H)|^2\frac{|\widehat k(\xi_H+\xi^\perp)|^2
}{|\widehat k_H(\xi_H)|^2}\d\xi^\perp \d\xi_H \\
&=\int_{\hat G/H^\perp} A(\xi_H)|\widehat g(\xi_H)|^2 \d\xi_H .
\end{align*}
Consequently, if $A\in L^\infty(\widehat G/H^\perp)$, then 
\begin{align*}
    \Vert M_Wg \Vert_{L^2(G)}^2 \leq \Vert A\Vert_{L^\infty(\hat G/H^\perp)}
\Vert g \Vert_{L^2(H)}^2 .
\end{align*} 
Thus the same Fourier formula defines a bounded operator $M:L^2(H)\to L^2(G)$, and 
\begin{align*}
    \Vert M \Vert_{L^2(H)\to L^2(G)}^2 \leq \Vert A\Vert_{L^\infty(\hat G/H^\perp)}.
\end{align*}
It remains to prove equality in the operator norm. Set
\begin{align*}
    S:=
\esssup_{\xi_H\in\hat G/H^\perp}
A(\xi_H).
\end{align*}
For $\varepsilon>0$, the set 
\begin{align*}
    E_\varepsilon
:=
\{\xi_H\in\widehat G/H^\perp\mid A(\xi_H)>S-\varepsilon\}
\end{align*}
has positive Haar measure. By regularity of Haar measure, we may choose a measurable subset $F_\varepsilon \subseteq E_\varepsilon$ with 
$0<\mu(F_\varepsilon)<\infty$. 
Define $\hat{g}_{\varepsilon}:\hat{G}/H^\perp \to \R$ by  
\begin{align*}
    \hat g_\varepsilon(\xi_H)
:=
\frac{\chi_{F_\varepsilon}(\xi_H)}
{\sqrt{\mu(F_\varepsilon)}} .
\end{align*}
Then $\Vert g_\varepsilon \Vert_{L^2(H)}=1$, and 
\begin{align*}
    \Vert Mg_\varepsilon \Vert_{L^2(G)}^2
=
\int_{\hat G/H^\perp}
A(\xi_H)|\widehat g_\varepsilon(\xi_H)|^2 \d\xi_H \
=
\frac{1}{\mu(F_\varepsilon)}
\int_{F_\varepsilon}
A(\xi_H) \d\xi_H
>
S-\varepsilon .
\end{align*}
Letting $\varepsilon\downarrow 0$, we obtain
\begin{align*}
    \Vert M\Vert_{L^2(H)\to L^2(G)}^2
=
S.
\end{align*}
If $\mathcal N(H)$ is dense in $L^2(H)$, then any bounded extension of $M_W$ to $L^2(H)$ is unique. 

The formula for the lower bound is proved analogously. Let 
\begin{align*}
  I:=    \essinf_{\xi_H \in \hat{G}/H^\perp} A(\xi_H).
\end{align*}
Since $A\geq I$ almost everywhere, every $g\in L^2(H)$ with
$\Vert g \Vert_{L^2(H)}=1$ satisfies 
\begin{align*}
    \Vert Mg \Vert_{L^2(G)}^2
=
\int_{\hat G/H^\perp}
A(\xi_H)|\hat g(\xi_H)|^2 \d\xi_H
\geq I.
\end{align*}
Conversely, for $\varepsilon>0$, choose a measurable set
$F_\varepsilon\subseteq\{\xi_H \in \hat{G}/H^\perp \, \vert \, A(\xi_H)<I+\varepsilon\}$ with
$0<\mu(F_\varepsilon)<\infty$, and define $\hat{g}_\varepsilon: \hat{G}/H^\perp \to \R$ by
\begin{align*}
    \hat g_\varepsilon
:=
\frac{\chi_{F_\varepsilon}}{\sqrt{\mu(F_\varepsilon)}}.
\end{align*}
Then $\Vert g_\varepsilon \Vert_{L^2(H)}=1$ and 
\begin{align*}
    \Vert Mg_\varepsilon\Vert_{L^2(G)}^2
<
I+\varepsilon .
\end{align*}
Letting $\varepsilon\downarrow 0$ yields 
\begin{align*}
    \inf_{\Vert g \Vert_{L^2(H)}=1}
\Vert Mg\Vert_{L^2(G)}^2
=
I.
\end{align*}
Finally, the minimal-norm extension operator satisfies 
\begin{align*}
    \Vert M_Wg \Vert_{\mathcal N(G)}
= 
\Vert g\Vert_{\mathcal N(H)},
\qquad g\in\mathcal N(H), 
\end{align*}
see \cite[Theorem 10.46]{wendland2004scattered}. Hence $\Vert M_W\Vert_{\mathcal{N}(H) \to \mathcal{N}(G) } =1$. 
Applying Corollary \ref{cor:Interpolation} to $M$ and $M_W$ yields the assertion.
\end{proof}

\begin{proof}[Proof of Theorem \ref{Theorem:MainCompactResult}]
    Since $G$ is compact, the dual group $\hat G$ is discrete, see Lemma \ref{lemma:compactdiscrete}. Hence the integrals over $H^\perp$ in Theorem \ref{Theorem:MainResult} become sums. The restricted spectral density of $Y=X|_H$ is therefore 
    \begin{align*}
        \hat{k}_{Y}(\xi_H)= \sum_{\xi^\perp \in H^\perp} \hat{k}_X(\xi_H+\xi^\perp), \qquad \xi_H \in \hat{G}/H^\perp.
    \end{align*} 
    Applying Theorem \ref{Theorem:MainResult} gives the Fourier-side formula 
    \begin{align*}
        \widehat{M_Wg} (\xi) = \hat{g}(\xi_H) \frac{\hat{k}_X(\xi)}{\hat{k}_Y(\xi_H)}, \qquad\xi_H:=\xi+H^\perp.
    \end{align*}
It remains to check the contraction bound. For fixed $\xi_H \in \hat{G}/H^\perp$, set
\begin{align*}
    a_{\xi^\perp} := \hat{k}_X(\xi_H+\xi^\perp), \qquad \xi^\perp \in H^\perp.
\end{align*}
Since $\hat{k}_X \geq 0$, the sequence $(a_{\xi^\perp})_{\xi^\perp \in H^\perp}$ is non-negative and summable. Therefore 
\begin{align*}
    \sum_{\xi^\perp \in H^\perp} a_{\xi^\perp}^2 \leq \left( \sum_{\xi^\perp \in H^\perp} a_{\xi^\perp} \right)^2. 
\end{align*}
Consequently, 
\begin{align*}
    \frac{
\sum_{\xi^\perp\in H^\perp}
|\hat k_X(\xi_H+\xi^\perp)|^2
}{
\left|
\sum_{\xi^\perp\in H^\perp}
\hat k_X(\xi_H+\xi^\perp)
\right|^2
}
\leq 1
\end{align*}
whenever the denominator is non-zero, while the quotient is interpreted as zero when the denominator vanishes. Thus the function $A$ from Theorem \ref{Theorem:MainResult} satisfies $A\leq 1$, and the asserted $L^2(H)\to L^2(G)$ bound follows. Applying Corollary \ref{cor:Interpolation} also implies the bounds for $M_\theta:\mathcal{N}(H)_\theta \to \mathcal{N}(G)_\theta$. 

If $W_Y$ is dense in $L^2(H)$, then the bounded extension of $M_W$ to $L^2(H)$ is unique. The identity 
\begin{align*}
    MY=\E(X \vert Y)
\end{align*}
then follows from the general relation between the minimal-norm extension operator and Gaussian conditional expectations, see \cite[Example 42]{WinkleAnalysis}. 
\end{proof}

\newpage

\end{appendices}

\end{document}